\documentclass[11pt, english]{article}

\usepackage[margin= 2 cm,bottom=19mm, top= 18mm]{geometry}

\usepackage[square,sort,comma,numbers]{natbib}
\usepackage{amsthm}
\usepackage{amsmath}
\usepackage{amssymb}
\usepackage{setspace}
\usepackage{mathtools}
\usepackage{graphicx}
\graphicspath{ {./images/} }
\usepackage[hidelinks]{hyperref}
\usepackage{cleveref}
\usepackage{hyperref}
\usepackage{enumitem}
\usepackage{framed}
\usepackage{subcaption}
\usepackage{xspace}
\usepackage{thmtools} 
\usepackage{thm-restate}

\usepackage{thmtools}
\usepackage{thm-restate}

\usepackage{floatrow}

\usepackage{tikz}
\usepackage{mathdots}
\usepackage{xcolor}
\usepackage{diagbox}
\usepackage{colortbl}
\usepackage[absolute,overlay]{textpos}

\usetikzlibrary{calc}
\usetikzlibrary{decorations.pathreplacing}
\usetikzlibrary{positioning,patterns}
\usetikzlibrary{arrows,shapes,positioning}
\usetikzlibrary{decorations.markings}

\tikzstyle{edge}=[very thick]
\definecolor{bostonuniversityred}{rgb}{0.8, 0.0, 0.0}
\definecolor{arsenic}{rgb}{0.23, 0.27, 0.29}
\tikzstyle{diredge}=[postaction={decorate,decoration={markings,
		mark=at position .95 with {\arrow[scale = 1]{stealth};}}}]
\tikzset{
    arrow/.style={decoration={markings, mark=at position 0.7 with
    {\fill(-0.09*#1,-0.03*#1) -- (0,0) -- (-0.09*#1,0.03*#1) -- cycle;}}, postaction={decorate}},
    arrow/.default=1
}
\tikzset{
    arow/.style={decoration={markings, mark=at position 1 with
    {\fill(-0.09*#1,-0.03*#1) -- (0,0) -- (-0.09*#1,0.03*#1) -- cycle;}}, postaction={decorate}},
    arow/.default=1
}
\tikzset{
    arrrow/.style={decoration={markings, mark=at position 0.9 with
    {\fill(-0.09*#1,-0.03*#1) -- (0,0) -- (-0.09*#1,0.03*#1) -- cycle;}}, postaction={decorate}},
    arow/.default=1
}

\newcommand{\fitellipsis}[2] 
{\draw [fill=white]let \p1=(#1), \p2=(#2), \n1={atan2(\y2-\y1,\x2-\x1)}, \n2={veclen(\y2-\y1,\x2-\x1)}
    in ($ (\p1)!0.5!(\p2) $) ellipse [ x radius=\n2/2+0cm, y radius=1.1cm, rotate=\n1];
}
\newcommand{\Fitellipsis}[2] 
{\draw [fill=white]let \p1=(#1), \p2=(#2), \n1={atan2(\y2-\y1,\x2-\x1)}, \n2={veclen(\y2-\y1,\x2-\x1)}
    in ($ (\p1)!0.5!(\p2) $) ellipse [ x radius=\n2/2+0cm, y radius=1.4cm, rotate=\n1];
}

\theoremstyle{plain}

\newtheorem*{thm*}{Theorem}
\newtheorem{thm}{Theorem}[section]
\Crefname{thm}{Theorem}{Theorems}

\newtheorem*{lem*}{Lemma}
\newtheorem{lem}[thm]{Lemma}
\Crefname{lem}{Lemma}{Lemmas}

\newtheorem*{claim*}{Claim}

\crefname{claim}{Claim}{Claims}
\Crefname{claim}{Claim}{Claims}

\newtheorem{prop}[thm]{Proposition}
\Crefname{prop}{Proposition}{Propositions}

\Crefname{remar}{Remark}{Remarks}

\newtheorem{cor}[thm]{Corollary}
\crefname{cor}{Corollary}{Corollaries}

\newtheorem*{conj*}{Conjecture}
\newtheorem{conj}[thm]{Conjecture}
\crefname{conj}{Conjecture}{Conjectures}

\Crefname{qn}{Question}{Questions}

\newtheorem{obs}[thm]{Observation}
\Crefname{obs}{Observation}{Observations}

\Crefname{ex}{Example}{Examples}

\theoremstyle{definition}

\Crefname{prob}{Problem}{Problems}

\newtheorem{defn}[thm]{Definition}
\Crefname{defn}{Definition}{Definitions}

\theoremstyle{remark}

\renewenvironment{proof}[1][]{\begin{trivlist}
\item[\hspace{\labelsep}{\bf\noindent Proof#1.\/}] }{\qed\end{trivlist}}

\newcommand{\remove}[1]{}

\newcommand{\eps}{\varepsilon}

\title{\vspace{-1 cm}
Proof of Grinblat's conjecture on rainbow matchings in multigraphs}
\date{}

\author{
David Munh\'a Correia\thanks{
Department of Mathematics, ETH, Z\"urich, Switzerland. Research supported in part by SNSF grant 200021\_196965.
\newline
\emph{Emails}: \textbf{\{david.munhacanascorreia, benjamin.sudakov\}@math.ethz.ch}.
}
\and
Benny Sudakov\footnotemark[1]}

\begin{document} 
\maketitle
\begin{abstract}
Many well-known problems in Combinatorics can be reduced to finding a large rainbow structure in a certain edge-coloured multigraph. Two celebrated examples of this are Ringel's tree packing conjecture and Ryser's conjecture on transversals in Latin squares. In this paper, we answer such a question raised by Grinblat twenty years ago. Let an \textit{$(n,v)$-multigraph} be an $n$-edge-coloured multigraph in which the edges of each colour span a disjoint union of non-trivial cliques that have in total at least $v$ vertices. Grinblat conjectured that for all $n \geq 4$, every $(n,3n-2)$-multigraph contains a rainbow matching of size $n$. Here, we prove this conjecture for all sufficiently large $n$.
\end{abstract}
\section{Introduction}
A \emph{rainbow copy} of a graph $H$ in an edge-coloured graph $G$ is a subgraph of $G$ isomorphic to $H$ whose edges have different colours. 
There are many well-known problems in Combinatorics that can be reduced to finding a large rainbow structure in a certain edge-coloured multigraph.
One example of this is the famous conjecture of Ringel \cite{ringel1963theory} from 1963 stating that the edges of the complete graph $K_{2n+1}$ can be decomposed by copies of any tree on $n$ vertices. K\"{o}tzig \cite{rosa1966certain} noticed that this can be reduced to showing that a certain edge-colouring of $K_{2n+1}$ contains a rainbow copy of any tree on $n$ vertices. Recently, this problem and some other questions about finding rainbow trees were resolved in \cite{MPS2020, montgomery2020proof}. 

Another celebrated problem in the area involves transversals in Latin squares, the study of which dates back to the work of Euler in the 1700s. A Latin square of order $n$ is an $n \times n$ array filled with $n$ symbols such that no symbol appears more than once in a row or column and a transversal is a set of entries such that no two of them have a symbol, row or column in common. The Ryser-Brualdi-Stein conjecture (\cite{brualdi1991combinatorial, hall1955complete, ryser1967neuere, stein1975transversals}) states that every Latin square contains a transversal using all but at most one symbol. It is not difficult to see that a Latin square of order $n$ is actually equivalent to a proper $n$-edge-colouring of the complete bipartite graph $K_{n,n}$ and a transversal is now a rainbow matching in this graph. Thus, the conjecture states that there is always a rainbow matching of size $n-1$ in such a graph. Although this still remains open, the problem has atracted a lot of attention over the last 50 years (see, e.g., \cite{keevash2020new} and its references). Improving previous bounds from \cite{hatami2008lower, woolbright1978n}, the best known result towards this conjecture was recently obtained by Keevash, Pokrovskiy, Sudakov and Yepremyan \cite{keevash2020new} who showed that there is always a rainbow matching of size $n - O \left(\frac{\log n}{\log \log n} \right)$. 

There are now many variants and generalisations of the Ryser-Brualdi-Stein conjecture. One of these is the Aharoni-Berger conjecture \cite{aharoni2009rainbow}, which states that every edge-coloured bipartite multigraph with $n$ colours, each consisting of a matching of size $n+1$, contains a rainbow matching using all the colours. Note that indeed this implies the Ryser-Brualdi-Stein conjecture since any properly $n$-edge-coloured $K_{n,n}$ can be transformed into such a graph by adding to it a disjoint edge repeated in each one of the $n$ colours. This problem has been extensively studied (see, e.g., \cite{aharoni2015generalization, aharoni2017representation, clemens2015improved, kotlar2014large, pokrovskiy2015rainbow}) and the conjecture was shown to hold asymptotically in \cite{pokrovskiy2018approximate} (see also \cite{CPS2021+} for a very short proof).  

In this paper, we will consider another open problem in this area. An \textit{$(n,v)$-multigraph} is an $n$-edge-coloured multigraph in which each \emph{colour class}, i.e., the graph formed by the edges of each colour, forms a disjoint union of non-trivial cliques that in total have at least $v$ vertices. These can be seen as a generalisation of the type of edge-coloured multigraphs mentioned earlier. Indeed, note that the Aharoni-Berger conjecture is equivalent to the statement that every bipartite $(n,2n+2)$-multigraph contains a rainbow matching of size $n$. Twenty years ago, Grinblat \cite{grinblat2002algebras} raised the question of how large $v$ should be so that every $(n,v)$-multigraph contains a rainbow matching using all the colours. In fact, initially, Grinblat's question was formulated as a measure-theoretic problem in the context of his work on algebras of sets, in which he looked at sufficient conditions for a family of algebras over a set $X$ to cover the whole power set $\mathcal{P}(X)$. His question was later reformulated as a graph theoretic problem and gained attention of combinatorialists.

Let us note first that if $v \geq 4n-3$ we can greedily find a rainbow matching of size $n$. Indeed, given a rainbow matching $M$ of size at most $n-1$, if any colour not used in it has at least $4n-3$ vertices in its colour class, which is a disjoint union of non-trivial cliques, then note that it must have an edge outside $M$ and thus, we can add it and get a larger rainbow matching. Grinblat \cite{grinblat2002algebras} first showed that $v$ should be larger than $3n-3$ and is at most $4n - \lfloor \frac{n+3}{2} \rfloor$. The lower bound follows by considering a disjoint union of $n-1$ triangles, whose edges are repeated in each one of the $n$ colours. Note that a matching cannot have more than one edge from each triangle and so, there is no matching of size $n$. For the smaller values of $n=2,3$, one can do slightly better. E.g., for $n=3$, consider two disjoint copies of a proper edge-colouring of $K_4$ with three colours. Nevertheless, Grinblat \cite{grinblat2002algebras} conjectured  the following.
\begin{conj}
For all $n \geq 4$, every $(n,3n-2)$-multigraph contains a rainbow matching of size $n$.
\end{conj}
\noindent In \cite{grinblat2015families}, he subsequently improved the upper bound to $10n/3 + \sqrt{2n/3}$. Later, Nivasch and Omri \cite{nivasch2017rainbow} lowered it to $16n/5 + O(1)$. Finally, Clemens, Ehrenm\"{u}ller and Pokrovskiy \cite{clemens2017sets} gave the first asymptotic proof of the conjecture showing that there is a constant $c$ such that every $ \left(n,3n+c \sqrt{n} \right)$-multigraph has a rainbow matching of size $n$. In this paper, we resolve Grinblat's conjecture for all sufficiently large $n$.
\begin{thm}\label{mainthm}
There is $n_0$ such that for all $n \geq n_0$, every $(n,3n-2)$-multigraph contains a rainbow matching of size $n$.
\end{thm}
\noindent In the next section we will give some notation, preliminary lemmas and a proof outline. We will use these to prove Theorem \ref{mainthm} in Section 3. The last section of the paper contains some concluding remarks and further questions.
\section{Preliminaries}
We begin by first giving some notation and definitions. Throughout the paper we use standard graph-theoretic notation. We will often consider an edge-coloured multigraph $G$. For a colour $c$ we will say that an edge $e$ of colour $c$ is \textit{$c$-coloured} and that $c(e) = c$. The edge $e$ is naturally associated to a pair of vertices $xy$ which are its endpoints and we will also then say that the pair $xy$ is $c$-coloured. The edge $e$, or the pair $xy$, will be said to be \textit{repeated in at least $t$ colours} if there are at least $t$ different colours $c'$ such that $xy$ is $c'$-coloured. More generally, for a set of colours $C$, we will say that the edge $e$ (or the pair $xy$) is $C$-coloured if $c(e) \in C$. For two sets of vertices $X,Y$, we will let $E[X,Y]$ denote the set of edges which have one endpoint in $X$ and the other in $Y$. For a set of edges $F$, we will let $V(F)$ denote the set of endpoints of these edges and $C(F)$ denote the set of colours $c$ with a $c$-coloured edge in $F$. Given a matching $M$ and a set $A \subseteq V(M)$, we let $m(A)$ denote the set of opposite vertices in $M$ of the vertices in $A$ - usually the small $m$ notation will be unambiguous.  Also, in some situations, it will be clear that although we are referring to an edge $e$, we are instead considering its endpoints $xy$ - for example, we might refer to an edge-set of the form $E[e,X]$, where we are obviously considering the set $E \left[\{x,y\},X \right]$. Furthermore, we might say that $xy \in E[X,Y]$ where we mean that any edge with those endpoints belongs to that edge-set. We also make the following definition. 
\begin{defn}
For a matching $M$, a \emph{horn in $M$} is an edge $e \in M$ such that there exist two disjoint edges $e_1,e_2 \in E[e,V \setminus V(M)]$. It is a \emph{$c$-horn} if both edges $e_1,e_2$ are $c$-coloured and it is a \emph{$C$-rainbow horn} if there exist two distinct $c_1,c_2 \in C$ such that $e_1$ is $c_1$-coloured and $e_2$ is $c_2$-coloured. 
\end{defn}
\noindent Finally, we make a simple observation which we will refer to throughout the paper. Note that given a $(n,v)$-multigraph, we can always assume that each monochromatic clique is either a $K_3$ or a $K_2$, since every clique can be partitioned into disjoint edges and at most one triangle covering the same set of vertices.
\subsection{A few simple lemmas}
In this section, we will give some very simple lemmas which will be used throughout the paper. First, let us note the following.
\begin{obs}\label{observ}
Let $M$ be matching, $e \in M$ and suppose there is a set $C$ of three colours such that $e$ is a $c$-horn for two colours from $C$ and there is some edge in $E[e,V \setminus V(M)]$ coloured in the third colour from $C$. Then, $e$ is a $C$-rainbow horn.
\end{obs}
\noindent It is quick to check the above observation and further, as a trivial corollary, note the following. \begin{lem}\label{hornslemma}
Let $N$ be a matching and $C$ a set of colours such that for each $c \in C$ there are at least $k$ many $c$-horns in $N$. If $k|C| > 2|N|$, there exists a $C$-rainbow horn in $N$.
\end{lem}
\begin{proof}
Suppose there is no $C$-rainbow horn. By the previous observation, then no edge in $N$ is a $c$-horn for at least three colours $c \in C$. Therefore, since there are $k$ many $c$-horns for each colour $c$, we must have $k|C| \leq 2|N|$, which is a contradiction.
\end{proof}
\noindent We also give a lemma which, in various situations, will essentially allow us to forget about the various monochromatic triangles which can appear in $(n,v)$-multigraphs and only look at each colour class as a matching. 
\begin{lem}\label{triangles}Let $M$ be a matching and $A \subseteq V(M)$ be such that $A \cap m(A) = \emptyset$.   Let $H$ be a disjoint union of non-trivial cliques with no edges contained in $A \cup (V \setminus V(M))$ and with in total at least $2|M|+s$ vertices. Then, there exists a matching in $H$ of size $s$ consisting of edges with one endpoint in $V(M) \backslash (A \cup m(A))$ and the other in $A \cup (V \setminus V(M))$.
\end{lem}
\begin{proof}
Define the set $S := V(M) \setminus A$ so that there is no edge of $H$ contained outside $S$. Note that $H$ has at least $|A|+s$ vertices outside $S$ and so, there is an edge of $H$ from each such vertex to $S$. Since the graph $H$ is a disjoint union of non-trivial cliques, these edges must be pairwise disjoint, as otherwise, there would be an edge of $H$ connecting their endpoints outside $S$. This then gives a matching in $H$ contained in $E[S,V \setminus S]$ of size at least $|A|+s$. Since at most $|A|$ edges of this matching intersect $m(A)$, we are done.
\end{proof}
We will finally need a standard probabilistic concentration inequality, which can be found in most probabilistic textbooks (e.g., \cite{hoeffding}).
\begin{lem}\label{chernoff}
Let $X$ be the sum of independent random variables $X_1, \ldots, X_n$ such that each $0 \leq X_i \leq k $. Then, for all $0 < \eps < 1$,
$$\mathbb{P} \left( |X - \mathbb{E}[X]| > \eps \mathbb{E}[X] \right) \leq 2 e^{-\eps^2 \mathbb{E}[X]/3k^2}$$
\end{lem}
\subsection{Auxiliary matchings}
For a rainbow matching $M$, a \emph{$t$-auxiliary matching for $M$} is a matching $N \subseteq E[V \setminus V(M),V(M)]$ such that each one of its edges is repeated in at least $t$ colours which are not used in $M$ and no two of its edges intersect the same edge in $M$. We will let $M_N \subseteq M$ denote the set of edges of $M$ which intersect an edge of $N$. For $e \in M_N$, we will let $x_e$ denote the endpoint of $e$ which is contained in $V(N)$, $m(x_e)$ denote its opposite vertex in $M$ (as usual) and $v_e$ denote its opposite vertex in $N$. Finally, we let $C_N$ denote the set of colours used in $M_N$, that is, the set $C(M_N)$. The reader might want to refer to Figure \ref{am} for an illustration.

\begin{figure}[h]
\RawFloats
\begin{minipage}[t]{0.5\textwidth}
\centering
\captionsetup{width=\textwidth}

\captionsetup{width=0.8\textwidth}

\begin{tikzpicture}[scale =1, xscale=1, yscale=1]

\draw[rounded corners] (-1,-1.6) rectangle (1,1.2);
\node[scale=1] (M) at(0,1.6) {$M$};



\draw (-0.8,0.8) -- (0.8,0.8);
\draw (-0.8,-0.4) -- (0.8,-0.4);
\draw (-0.8,-0.8) -- (0.8,-0.8);
\draw (-0.8,-1.2) -- (0.8,-1.2);


\draw[red] (-0.8,0.8) -- (-1.8,0.8);
\draw[red] (-0.8,-0.4) -- (-1.8,-0.4);

\draw[red] (-0.8,0.8) -- (-1.8,0.8);
\draw[red] (-0.8,-0.4) -- (-1.8,-0.4);

\draw[red] (-1.3,0.8) ellipse (0.5 and 0.1); 
\draw[red] (-1.3,-0.4) ellipse (0.5 and 0.1);

\draw[red] (-1.3,0.8) ellipse (0.5 and 0.1); 
\draw[red] (-1.3,-0.4) ellipse (0.5 and 0.1); 

\draw[red, rounded corners] (-2,-0.6) rectangle (-0.6,1);
\draw[red, rounded corners] (-2,-0.6) rectangle (-0.6,1);
\node[red, scale=1] (N) at(-2.5,1.2) {$N$};
\node[red, scale=1] (N) at(-2.5,1.2) {$N$};

\draw[thick] (-0.8,0.4) -- (0.8,0.4);
\draw[thick] (-0.8,0.4) -- (0.8,0.4);
\draw[red] (-1.3,0.4) ellipse (0.5 and 0.1); 
\draw[red] (-0.8,0.4) -- (-1.8,0.4);
\draw[red] (-1.3,0.4) ellipse (0.5 and 0.1); 
\draw[red] (-0.8,0.4) -- (-1.8,0.4);
\node[scale=3] at(-0.8,0.4) {$.$};
\node[scale=3] at(0.8,0.4) {$.$};
\node[scale=3] at(-1.8,0.4) {$.$};
\node[scale=1] at(0,0.6) {$e$};
\draw[arow=1.6,blue]  (0.9,0.5) -- (2,1.8);
\draw[arow=1.6,blue]  (-0.9,0.3) -- (-2.5,-1);
\draw[arow=1.6,blue]  (-1.9,0.35) -- (-3,0);
\node[scale=0.8] at(-3.3,-0.1) {$v_e$};
\node[scale=0.8] at(-2.7,-1.2) {$x_e$};
\node[scale=0.8] at(2.3,2.1) {$m(x_e)$};
\node[scale=1] at(0,0.1) {$.$};
\node[scale=1] at(0,0) {$.$};
\node[scale=1] at(0,-0.1) {$.$};

\node[scale=1] at(-1.3,-0.1) {$.$};

\node[scale=1] at(-1.3,0.1) {$.$};
\node[scale=1] at(-1.3,0) {$.$};
\draw [decorate,decoration={brace,  amplitude=5pt,raise=2pt},yshift=0pt]($(1.15,1)$) -- ($(1.15,-0.6)$);
\node[scale=1] at(2,0.4) {$M_N$};

\end{tikzpicture}

\caption{A $3$-auxiliary matching for $M$.}
\label{am}
\end{minipage}\hfill
\begin{minipage}[t]{0.5\textwidth}
\centering
\captionsetup{width=\textwidth}
\begin{tikzpicture}[scale = 0.95, xscale=1, yscale=1]

\draw[blue!100] (-0.8,2.4) -- (1,2.4);
\draw[blue!100] (-0.8,2.4) -- (1,2.4);
\draw[green!100] (-0.8,1.6) -- (1,1.6);
\draw[green!100] (-0.8,1.6) -- (1,1.6);

\draw (-0.8,0.8) -- (1,0.8);
\draw (-0.8,0.8) -- (1,0.8);
\draw (-0.8,0) -- (1,0);
\draw (-0.8,0) -- (1,0);
\draw (-0.8,-0.8) -- (1,-0.8);
\draw (-0.8,-0.8) -- (1,-0.8);

\draw (-0.8,0.8) -- (1,0.8);
\draw (-0.8,0.8) -- (1,0.8);
\draw (-0.8,0) -- (1,0);
\draw (-0.8,0) -- (1,0);
\draw (-0.8,-0.8) -- (1,-0.8);
\draw (-0.8,-0.8) -- (1,-0.8);

\draw[red!100] (-1.3,1.6) ellipse (0.5 and 0.2); 
\draw[red!100] (-1.3,1.6) ellipse (0.5 and 0.08); 
\draw[red!100] (-1.3,2.4) ellipse (0.5 and 0.2); 
\draw[red!100] (-1.3,2.4) ellipse (0.5 and 0.08); 
\draw[red!100] (-1.3,0.8) ellipse (0.5 and 0.2); 
\draw[red!100] (-1.3,0.8) ellipse (0.5 and 0.08); 
\draw[red!100] (-1.3,0) ellipse (0.5 and 0.2); 
\draw[red!100] (-1.3,0) ellipse (0.5 and 0.08); 

\draw[red!100] (-1.3,-0.8) ellipse (0.5 and 0.2); 
\draw[red!100] (-1.3,-0.8) ellipse (0.5 and 0.08); 

\draw[green!=100] (-1.8,-0.8) -- (1,0);
\draw[green!=100] (-1.8,-0.8) -- (1,0);
\draw[green!=100] (-1.8,-0.8) -- (1,0);
\draw[green!=100] (-1.8,-0.8) -- (1,0);
\draw[blue!=100] (-0.8,-0.8) -- (1,0.8);
\draw[blue!=100] (-0.8,-0.8) -- (1,0.8);
\draw[blue!=100] (-0.8,-0.8) -- (1,0.8);
\draw[blue!=100] (-0.8,-0.8) -- (1,0.8);

\node[scale=3] at(1,0.8) {$.$};
\node[scale=3] at(1,0) {$.$};
\node[scale=3] at(-0.8,-0.8) {$.$};
\node[scale=3] at(-1.8,-0.8) {$.$};
\node[scale=1] at(0.1,-1.1) {$f$};

\end{tikzpicture}
\caption{A $C_N$-rainbow horn in $N \cup (M \setminus M_N)$.}
\label{ill}
\end{minipage}
\end{figure}
\noindent As a first simple observation, we will show the following lemma, which will be the basis of most of the local considerations done throughout the paper. Indeed, a particular case of the lemma is Observation \ref{observ}.
\begin{lem}\label{lema}
Let $M$ be a maximal rainbow matching, $C_0$ be the set of colours not in $C(M)$,  $N$ be a $t$-auxiliary matching for $M$ and $M' \subseteq M$ be a set of size less than $t/2 - 1$. Then there is no $(C_0 \cup C_N)$-coloured rainbow matching of size $|M'|+1$ disjoint to the matching $\{v_e x_e: e \in M_N \setminus M'\} \cup (M \setminus (M_N \cup M'))$.
\end{lem}
\begin{proof}
Suppose otherwise and let $L$ be such a $\left(C_0 \cup C_N \right)$-coloured rainbow matching. Let $e_1, \ldots, e_i$ denote the edges of $M_N \setminus M'$ which intersect edges of $L$ and let $e'_1, \ldots, e'_j$ denote the edges of $M_N \setminus (M' \cup \{e_1, \ldots, e_i\})$ whose colours are represented in $L$. By the definition of a $t$-auxiliary matching and since $|C_0 \cap C(L)| + i + j \leq (|L|-j) + i + j \leq 2|L| = 2|M'| + 2 < t$, note that we can pick distinct colours in $C_0 \setminus C(L)$ for the edges $v_{e_l}x_{e_l}, v_{e'_l}x_{e'_l}$ so that $$(M \setminus (\{e_1, \ldots, e_i,e'_1, \ldots, e'_j\} \cup M')) \cup L \cup \{v_{e_1}x_{e_1}, \ldots, v_{e_i}x_{e_i}, v_{e'_1}x_{e'_1}, \ldots, v_{e'_j}x_{e'_j}\} $$
forms a larger rainbow matching than $M$, which is a contradiction.
\end{proof}
Next, we give two corollaries of the above lemma. In both of them, let $N$ be a $t$-auxiliary matching for $M$.
\begin{lem}\label{amclaim1}
Let $t \geq 5$ and $M$ be a maximal rainbow matching and $C_0$ be the set of colours not in $C(M)$. Then, there is no $(C_0 \cup C_N)$-coloured edge disjoint to the matching $N \cup (M \setminus M_N)$ and there is no $(C_0 \cup C_N)$-rainbow horn in the matching $N \cup (M \setminus M_N)$.
\end{lem}
\begin{lem}\label{amclaim2}
Let $t \geq 5$ and $M$ be a maximal rainbow matching and $C_0$ be the set of colours not in $C(M)$. Then, if $v_e z$ is a $(C_0 \cup C_N)$-coloured edge with $e \in M_N$ and $z \notin V(N) \cup V(M \setminus M_N) $, then $z = m(x_e)$ or the edge is of colour $c(e)$. 
\end{lem}
\noindent The first lemma can be easily checked by Lemma \ref{lema}. Indeed, note that the first part is a corollary of it with $M' = \emptyset$ and the second part follows when $M'$ consists of a single edge - the reader might want to refer to Figure \ref{ill} where the case of a rainbow horn of the form $v_f x_f$ for some $f \in M_N$ is illustrated. The second lemma is actually a corollary of the first - indeed, if $v_e z$ is a $((C_0 \cup C_N) \setminus \{c(e)\})$-coloured edge with $z \neq m(x_e)$ then the edges $e$ and $v_e z$ form a $(C_0 \cup C_N)$-coloured rainbow matching and so, the edge $v_ex_e$ is a $(C_0 \cup C_N)$-rainbow horn in the matching $N \cup (M \setminus M_N)$, thus contradicting Lemma \ref{amclaim1}.

We are now ready to give the two main facts about auxiliary matchings that we will need. The first will essentially show that maximal rainbow matchings have large auxiliary matchings.
\begin{lem}\label{amlemma}
For all $\delta > 0$ and $t$, there exist $\delta' > 0$ and $r$ such that the following holds for all sufficiently large $n$. Let $G$ be a $(n,\left(3-\delta' \right)n)$-multigraph, $M$ a maximal rainbow matching in $G$ and suppose $|M| < n-r$. Then, there is a $t$-auxiliary matching for $M$ of size at least $(1-\delta)n$. 
\end{lem}
\begin{proof}
Take $r := \max \big(10/ \delta', 2t/\sqrt{\delta'}\big)$ and let $\delta'$ be sufficiently small as a function of $t$ and $\delta$. Let $C_0$ denote the set of at least $r$ colours not used in $M$ and note that by Lemma \ref{amclaim1} (applied with an empty auxiliary matching), there is no $C_0$-coloured edge outside the matching $M$ and no $C_0$-rainbow horn in the same matching. Therefore, Lemma \ref{hornslemma} implies that at most half of the colours $c \in C_0$ are such that $M$ has at least $4|M|/r$ many $c$-horns. Let us then redefine $C_0$ in order to only include the at least $r/2$ many colours for which this does not occur. 

Now, each $c \in C_0$ has at least $(3-\delta')n$ vertices in its colour class and there is no $c$-coloured edge contained in $V \setminus V(M)$. Therefore, by Lemma \ref{triangles}, there is a $c$-coloured matching $M_c \subseteq E[V \setminus V(M),V(M)]$ of size $(1-\delta')n$. Furthermore, since $M$ has at most $4|M|/r$ many $c$-horns, $M_c$ must intersect at least $(1-\delta')n - 4|M|/r \geq (1-2\delta')n$ edges of $M$. We can now take a subset $C'_0 \subseteq C_0$ of $\lfloor t/\sqrt{\delta'} \rfloor \leq r/2$ colours for which there is a subset $M' \subseteq M$ of size at least $(1-2t\sqrt{\delta'})n$ so that every edge in $M'$ intersects all matchings $M_c$ with $c \in C'_0$. Let an edge $e \in M'$ be a \emph{bad edge} if it has an endpoint $u$ such that there are at least ten distinct vertices $z \notin V(M)$ with $uz \in \bigcup_{c \in C_0} M_c$. Let $M'_{\text{bad}}$ denote the set of bad edges and note the following.
\begin{claim*}
$|M'_{\text{bad}}| \leq 2 \delta' |M|$.
\end{claim*}
\begin{proof}
Suppose otherwise and let $V_{\text{bad}} \subseteq V(M'_{\text{bad}}) $ denote the set of endpoints implying the badness of the edges in $M'_{\text{bad}}$. Letting $A := m(V_{\text{bad}})$, we can check that there cannot exist a $C_0$-coloured edge contained in $A \cup (V \setminus V(M))$. Indeed, for contradiction sake, let $e$ be such an edge and let $\{e_1,e_2\} \subseteq M'_{\text{bad}}$ be a set containing all edges which intersect $e$. Since $e_1, e_2$ are bad edges, note that there exist edges $e'_1 \in E[e_1 \cap V_{\text{bad}}, V \setminus V(M)]$ and $e'_2 \in 
E[e_2\cap V_{\text{bad}},V \setminus V(M)] $ so that $\{e,e'_1,e'_2\}$ forms a $C_0$-coloured rainbow matching. In turn, then $(M \setminus \{e_1,e_2\}) \cup \{e,e'_1,e'_2\}$ is a larger rainbow matching than $M$.

We can now use Lemma \ref{triangles} and have that for each $c \in C_0$ there is a $c$-coloured matching $L_c \subseteq E[V(M \setminus M'_{\text{bad}}),A \cup (V \setminus V(M))]$ of size at least $(1-\delta')n \geq (1-\delta')|M|$. Since by assumption, $|M \setminus M'_{\text{bad}}| \leq (1-2\delta')|M|$, each $c \in C_0$ is such that there are at least $\delta' |M| > 2 |M|/ |C_0|$ edges $e \in M \setminus M'_{\text{bad}}$ which intersect two edges of $L_c$. Therefore, there is an edge $e \in M \setminus M'_{\text{bad}}$ for which there are at least three colours $c$ such that $L_c$ intersects $e$ in two edges. By using a similar argument as for Observation \ref{observ}, there must then exist a $C_0$-coloured rainbow matching of two edges going from $e$ to $A \cup (V \setminus V(M))$. In turn, similarly to the last paragraph, it is easy to check that this contradicts the maximality of $M$. 

\end{proof}
Given the above claim, let us now delete the bad edges from $M'$ so that we still have $|M'| \geq \left(1-4t\sqrt{\delta'} \right)n$. To finish, we make a final observation implied by the definition of a bad edge. Recall first that the set $C'_0 \subseteq C_0$ is such that each edge of $M'$ intersects all matchings $M_c$ with $c \in C'_0$. Now, since we have deleted all bad edges, it must be the case that for each edge $e \in M'$, there are at most $20t$ colours $c \in C'_0$ such that there is an edge $uz \in M_c$ with $u \in e$, $z \notin V(M)$ and such that $uz$ is repeated in at most $t$ many colours in $C'_0$. For each such edge $e$, let $C_e \subseteq C'_0$ denote this set of colours. There must then exist a colour $c \in C'_0$ which belongs to at most 
$$\frac{1}{|C'_0|} \cdot \sum_{e \in M'} |C_e| \leq \frac{1}{|C'_0|} \cdot 20t |M'| \leq 40 \sqrt{\delta'} \cdot |M'|$$
many sets $C_e$. Finally, note that the matching $N$ formed by the edges of $M_c$ which intersect those edges $e \in M'$ such that $c \notin C_e$ forms a $t$-auxiliary matching, which is of size at least
$$\left(1 - 40 \sqrt{\delta'} \right)|M'| \geq \left(1 - 40 \sqrt{\delta'} \right) \left(1-4t\sqrt{\delta'} \right)n \geq (1-\delta)n.$$
\end{proof}
Next, given a $t$-auxiliary matching $N$ for a rainbow matching $M$, define the subset $N_{\alpha} \subseteq N$ to be the set of edges $v_e x_e \in N$ such that the edge $v_e m(x_e)$ is repeated in at most $\alpha |C_N|$ many colours of $C_N$, and define $M_{N_{\alpha}}$ as expected. The lemma below will tell us that these sets cannot be very large.
\begin{lem}\label{amlemma2}
For every $\gamma > 0$ and $\alpha > 0$, the following holds for all sufficiently large $n$. Let $G$ be an $(n,\left(3-\gamma \right)n)$-multigraph and $M$ a maximal rainbow matching in $G$ with a $7$-auxiliary matching $N$ of size at least $(1-\gamma)|M|$. Then, $|N_{\alpha}| \leq \left(\frac{20 \gamma}{1- \alpha} \right) n $. 
\end{lem}
\begin{proof}
For contradiction sake, suppose otherwise. First, for each colour $c \in C_N$, delete the $c$-coloured edges which intersect the vertex $v_e$ where $e \in M_N$ is such that $c(e) = c$. Secondly, delete those $c$-coloured edges which intersect $V(M \setminus M_N)$. Note that after these deletions, we can delete some final $c$-edges so that $c$ is now a disjoint union of triangles and edges with in total at least $(3-\gamma)n - 3 - 3 \cdot |V(M \setminus M_N)| \geq (3-\gamma)n - 3 - 6 \gamma |M| \geq (3-8 \gamma)n$ many vertices. In turn, because of the first deletion process and the maximality of $M$, note that Lemma \ref{lema} implies that there is now no $C_N$-coloured edge contained in $V \setminus V(M)$ - indeed, suppose $e$ is such an edge and let $M' \subseteq M_N$ be the subset of at most two edges $f$ such that $v_f$ intersects $e$; then, since the first deletion process implies that none of these edges have the same colour as $e$, we have that $ M' \cup e$ forms a rainbow matching of size $|M'|+1$ which contradicts Lemma \ref{lema}.

Now, let $A := \{m(x_e) : e \in M_N\}$. By Lemma \ref{amclaim1}, there is no $C_N$-coloured edge contained in $A$ and again because of the first deletion process, Lemma \ref{amclaim2} implies that any $C_N$-coloured edge of the form $m(x_e)y$ for $y \notin V(M)$ must be such that $y = v_e$. Therefore, by the definition of $N_{\alpha}$, there are at most $\alpha|C_N|$ many $C_N$-coloured edges of that form for each $e \in M_{N_{\alpha}}$. A simple counting argument then gives the following.
\begin{claim*}
There exist at least $(1-\alpha) |C_N|/2$ many colours $c \in C_N$ for which there exists a subset $A_c \subseteq A$ of size at least $(1-\alpha) |N_{\alpha}|/2$ such that there are no $c$-coloured edges contained in $A_c \cup (V \setminus V(M))$. 
\end{claim*}
\begin{proof}
For each edge $e \in M_{N_{\alpha}}$, let $C_e \subseteq C_N$ denote the set of at least $(1-\alpha)|C_N|$ many colours which do not appear in the edge $v_e m(x_e)$. Let $C'_N \subseteq C_N$ denote the set of colours which belong to at least $(1-\alpha) |N_{\alpha}|/2$ many sets $C_e$. Then, we must have
$$(1-\alpha)|C_N||N_{\alpha}| \leq \sum_{e \in M_{N_{\alpha}}} |C_e| \leq |C_N| \cdot (1-\alpha) |N_{\alpha}|/2 + |C'_N| \cdot |N_{\alpha}|$$
and thus, $|C'_N| \geq (1-\alpha) |C_N|/2$. Now, note that from the previous discussion, we have that for each $c \in C_N$, the possible $c$-coloured edges contained in $A \cup (V \setminus V(M))$ must be of the form $v_e m(x_e)$ for some $e \in M_N$. Therefore, letting $A_c \subseteq A$ denote the set of vertices $m(x_e)$ with $e \in M_{N_{\alpha}}$ and $c \in C_e$, we have that there is no $c$-coloured edge contained in $A_c \cup (V \setminus V(M))$. We are then done since $|A_c| \geq (1-\alpha) |N_{\alpha}|/2$ for each $c \in C'_N$.
\end{proof}
Let $C'_N \subseteq C_N$ denote the set of colours given by the above claim. From Lemma \ref{triangles} and the second deletion process done at the start of the proof, there exists a $c$-coloured matching $L_c \subseteq E[V(M_N) \setminus (A_c \cup m(A_c)),A_c \cup (V \setminus V(M))]$ of size $(1-8\gamma)n$ for each $c \in C'_N$. Since $|A_c| \geq (1-\alpha) |N_{\alpha}|/2 \geq 10\gamma n$, there are at least $2 \gamma n$ edges of $M_N$ which intersect two edges of $L_c$. Since $|C'_N| \geq (1-\alpha) |C_N|/2 \geq (1-\alpha) |N_{\alpha}|/2  \geq 10\gamma n$ and $10 \gamma n{2 \gamma n \choose 50} > 3{n \choose 50}$ (provided that $n$ is large enough), it is easy to note that there exists three colours $c_1,c_2,c_3 \in C'_N$ such that there are at least 50 edges in $M_N$ each intersecting two edges of each $L_{c_i}$. By using the simple fact that a family of three matchings of size two contains a rainbow matching of size two, there then exist two of those colours, say $c_1,c_2$, for which there are at least 15 edges $e \in M_N$ each such that its endpoints are matched to $A \cup (V \setminus V(M))$ by one edge of $L_{c_1}$ and one edge of $L_{c_2}$. 

To finish, we claim that this contradicts the maximality of $M$. Indeed, let $e_1,e_2 \in M_N$ be the edges in $M$ of colours $c_1,c_2$ and let $S:= e_1 \cup e_2 \cup \{v_{e_1},v_{e_2} \}$. By the conclusion of the last paragraph, there must exist an edge $e \in M_N$ matched to $A \cup (V \setminus V(M))$ by edges $e'_1 \in L_{c_1}$ and $e'_2 \in L_{c_2}$ so that $e'_1,e'_2$ are disjoint to $S$. Let $e'_1$ be the edge with $m(x_e)$ as one of its endpoints. Note that its other endpoint must be a vertex $v_f$ for some $f \in M_N \setminus \{e_1,e_2\}$, as otherwise this edge, which is $C_N$-coloured, would be contained outside $V(N) \cup V(M \setminus M_N)$, which contradicts Lemma \ref{amclaim1}. Furthermore, if $f \neq e$, then the edges $f$ and $e'_1$ form a rainbow matching and thus, $x_f v_f$ is a $C_N$-rainbow horn in the matching $N \cup (M \setminus M_N)$, which also contradicts the same lemma - thus, $e'_1 = v_e m(x_e)$. Now, letting $e'_2 = x_e z$, note that if $z \in A$ or $z \in V \setminus (V(M) \cup V(N))$, we have that $v_e x_e$ is another rainbow horn which contradicts Lemma \ref{amclaim1}. Hence, $z = v_{f}$ for some $f \in M_N \setminus \{e_1,e_2\}$. Then, by taking $M'=\{e,f\} \subset M_N$ one can check that the rainbow matching $\{e'_1,e'_2,f\}$ contradicts Lemma \ref{lema}.
\end{proof}
To finish the section, we conclude with the following quick corollary of Lemmas \ref{amlemma} and \ref{amlemma2}.
\begin{cor}\label{corollary}
For all $\delta > 0$ and $t\geq 7$, there exist $\delta' > 0$ and $r$ such that the following holds for all sufficiently large $n$. Let $G$ be an $(n,\left(3-\delta' \right)n)$-multigraph, $M$ a maximal rainbow matching in $G$ and suppose that $|M| < n-r$. Then, there exists a $t$-auxiliary matching $N$ for $M$ of size at least $(1-\delta)n$ with $N_{1/4} = \emptyset$. 
\end{cor}
\begin{proof}
Let us assume without loss of generality that $\delta < 1/5$ and let $\delta_0 = \delta/41$ and note that from Lemmas \ref{amlemma} and \ref{amlemma2}, there is $\delta'$ and $r$ which implies the existence of a $t$-auxiliary matching $N'$ for $M$ of size at least $(1-\delta_0)n$ with $|N'_{1/2} | \leq 40 \delta_0 n$. Let $N := N' \setminus N'_{1/2}$ which is also a $t$-auxiliary matching, and of size at least $(1-41 \delta_0)n = (1-\delta)n$. Then, for each edge $e \in M_N$, which is then not in $M_{N'_{1/2}}$, we have, since $|C_{N'} \setminus C_N| = |N'_{1/2}|$, that the pair $m(x_e) v_e$ is repeated in at least $|N'|/2 - |N'_{1/2}| > (1/2 - \delta_0/2 - 40 \delta_0)n > n/4$ many $C_N$-colours. Therefore, $N_{1/4} = \emptyset$. 
\end{proof}
\subsection{An outline of the proof}
Although the proof of Theorem \ref{mainthm} will be relatively short, we give a brief overview of the main ideas. The first crucial thing to observe is that the Grinblat problem behaves differently when one restricts the edge-multiplicity of the $(n,v)$-multigraph. For example, it was shown in \cite{CPS2021+} (improving upon results from \cite{correia2020full}) that any $(n,2n+2m+o(n))$-multigraph with edge-multiplicity at most $m$ contains a rainbow matching using all the colours. This indicates that in general, lower multiplicity will make it easier to find such a rainbow matching. 

Our first step in the proof of Theorem \ref{mainthm} will be in this direction. We will show that in any $(n,v)$-multigraph with multiplicity at most $(1-\delta)n$ and with $v$ close enough to $3n$, we can always find a rainbow matching of size $n$. In order to prove this, we will first show that one can find a rainbow matching of size $n-f(\delta)$, which is a quick corollary of Lemmas \ref{amlemma} and \ref{amlemma2}, and then use the sampling trick introduced in \cite{CPS2021+} to transform this into a result giving a full rainbow matching. 

Next, given an $(n,3n-2)$-multigraph $G$ with some edges of multiplicity larger than $(1-\delta)n$, we iteratively delete $r$ such pairwise disjoint edges (along with their endpoints), for some large constant $r$ and small enough $\delta$ so that $r \leq 1/10\delta$. The resulting graph $G'$ is now not necessarily a $(n,3(n-r)-2)$-multigraph, but we can choose a set $C'$ of at least $(1-r\delta)n \geq 0.9n$ colours which were repeated in every edge belonging to the matching that was deleted. Note that each one of the colours $c \in C'$ does have at least $3(n-r)-2$ vertices in its colour class in $G'$. Now we only need to find a rainbow matching of size $n-r$ in $G'$ such that the colours not used in it belong to $C'$ - after we can choose colours from $C'$ for the edges of the deleted matching to get a rainbow matching of size $n$ in $G$. In fact, as $C'$ is very large and each colour class has size much larger than $2n$, we will be able to transform any rainbow matching of size $n-r$ into one that does not use colours from $C'$.

Hence, finding a rainbow matching of size $n-r$ in $G'$ is the main part of the proof. In fact, just from the lemmas given in this section, it is relatively easy to establish a bound which is only by one more than Grinblat's conjecture. Indeed, suppose we started with $G$ being a $(n,3n-1)$-multigraph, so that all colours in $C'$ have at least $3(n-r)-1$ vertices in their colour class in $G'$. Take a maximal rainbow matching in $G'$ with $|M| < n-r$. By Lemma \ref{amlemma}, there is a $7$-auxiliary matching $N'$ for $M$ of size, say, at least $0.9n$. Then $|C_{N'} \cap C'| \geq 0.8n$. 
So we can restrict to an auxiliary matching $N \subseteq N'$ of size $0.8n$ such that $C_N \subset C'$. Furthermore, by Lemmas \ref{triangles} and \ref{amclaim1}, for each colour $c \in C_N$, there is a $c$-coloured matching $M_c$ of size $|M|+2$ consisting of edges going from the matching $N \cup (M \setminus M_N)$ to the outside. Also, by Lemma \ref{amclaim1} there cannot exist a $C_N$-rainbow horn in the matching $N \cup (M \setminus M_N)$ and thus, Observation \ref{observ} implies that every edge in $N \cup (M \setminus M_N)$ must intersect at most $|C_N|+1$ many edges of $\bigcup_{c \in C_N} M_c$. By double-counting, we must then have $|C_N| \cdot (|M|+2) = \sum_{c \in C_N} |M_c| \leq |M| \cdot (|C_N| + 1) $ and so, $0.8n \leq |N| = |C_N| \leq |M|/2 \leq n/2$, which gives a contradiction. 

In order to prove Theorem \ref{mainthm}, we will need an additional observation which we discuss in the next section - it will consist of considering another 'auxiliary' matching.
\section{Proof of Theorem \ref{mainthm}}
As discussed in the proof outline, we will first need to prove a result concerning $(n,v)$-multigraphs with bounded edge-multiplicity.  
\begin{prop}\label{weakboundedmult}
For all $\delta > 0$ there exists $\delta' > 0$ and $r$ such that for all sufficiently large $n$, every $\left(n,(3-\delta')n \right)$-multigraph with multiplicity at most $(1-\delta)n$ contains a rainbow matching of size $n-r$.
\end{prop}
\begin{proof}
Without loss of generality, let us assume that $\delta < 1/100$. Let $G$ be a $\left(n,(3-\delta')n \right)$-multigraph with multiplicity at most $(1-\delta)n$ and, for sake of contradiction, let $M$ be a maximal rainbow matching of size at most $n-r$. By Lemma \ref{amlemma}, there exists a sufficiently small $\delta'$ which implies the existence of a $7$-auxiliary matching $N$ for $M$ of size at least $(1-\delta^2)n$. Additionally, by Lemma \ref{amlemma2} (with $\gamma=\delta^2$), it is the case that $|N_{1-\delta/2}| \leq 40 \delta n $. In turn, the multiplicity condition on $G$ implies that for each $e \in M_N$, there exist at most $(1-\delta)n \leq (1-\delta/2)(1-\delta^2)n \leq (1-\delta/2)|C_N|$ many $C_N$-colours repeated in the edge $v_e m(x_e)$. Thus, $N = N_{1-\delta/2}$, which leads to a contradiction since then $(1-\delta^2)n \leq |N| \leq 40\delta n $.
\end{proof}
We can now use the sampling trick introduced in \cite{CPS2021+} to transform the proposition above into the following theorem.
\begin{thm}\label{strongboundedmult}
For all $\delta > 0$ there exists $\delta' > 0$ such that for all sufficiently large $n$, every $\left(n,(3-\delta')n \right)$-multigraph with multiplicity at most $(1-\delta)n$ contains a rainbow matching of size $n$.
\end{thm}
\begin{proof}
Let $G$ be a $(n, (3-\delta')n)$-multigraph with multiplicity at most $(1-\delta)n$. For each colour $c$, let $t_c$ denote the number of triangles in its colour class and $l_c$ the number of edges, so that $3t_c + 2l_c \geq (3-\delta')n$. Let $S \subseteq V(G)$ be a random set obtained by choosing each vertex independently with probability $p = 2n^{-1/4}$. 

For each colour $c$, let $e_c(S)$ be the random variable counting the number of $c$-coloured edges contained in $S$, which is a sum of independent $[0,3]$-valued random variables. Moreover, every component of colour $c$ contributes an edge with probability at least $p^2$ and so, $\mathbb{E}[e_c(S)] \geq p^2(t_c+l_c) \geq p^2 n/2 = 2\sqrt{n}$. 
Thus, by Lemma \ref{chernoff}, we have $\mathbb{P}(X_c \leq \sqrt{n}) \leq e^{-\Omega(\sqrt{n})} = o(n^{-1})$, so by a union bound, we have that with probability $1-o(1)$, all $e_c(S) > \sqrt{n}$. Also for each colour $c$, let $v_c(S)$ denote the number of vertices in the initial colour class of $c$ which belong to $S$. Again using Lemma \ref{chernoff} it is easy to note that $\mathbb{P} \left(v_c(S) \geq 2(3-\delta')np \right) \leq e^{-\Omega_{\delta}(n^{3/4})} = o(n^{-1})$.
Moreover, since every vertex counted in $v_c(S)$ belongs to one clique of the original colour class, by deleting $S$ we might destroy at most $v_c(S)$ such cliques which cover at most $3v_c(S)$ vertices.
Therefore, by union bound, with probability $1-o(1)$, every colour class in $G-S$ is a disjoint union of non-trivial cliques covering in total at least 
$(1-6p)(3-\delta')n \geq (3-2\delta')n$ vertices. 

Now, fix a subset $S$ satisfying all the conditions discussed above. Then, provided that $\delta'$ is sufficiently small, Proposition \ref{weakboundedmult} implies that for some constant $r$ there is a rainbow matching $M$ in $G-S$ of size at least $n-r$ and every colour has at least $\sqrt{n}$ edges in $G[S]$. Let $C_0$ denote the set of colours not used in $M$. Since each colour class in $C_0$ has maximum degree two and at least $\sqrt{n} \gg 4r = 4|C_0|=2\cdot 2 \cdot |C_0|$ edges in $G[S]$, we can greedily find a rainbow matching $N \subseteq G[S]$ which uses all colours in $C_0$. As a result, $M \cup N$ is a full rainbow matching in $G$.
\end{proof}
Now that we have the desired bounded multiplicity result, we will next show that we can indeed find a rainbow matching of size $n-r$ in the situation discussed in the end of the proof outline.
\begin{lem}\label{keylemma}
Let $r$ be sufficiently large, $n$ sufficiently large in terms of $r$ and $G$ a $(n,3n-10r)$-multigraph such that there are at least $0.9n$ many colours with at least $3(n-r) - 2$ vertices in its colour class. Then, $G$ contains a rainbow matching of size $n-r$.
\end{lem}
\begin{proof}
Let $C'$ be the set of at least $0.9n$ many colours mentioned in the statement, let $M$ be a maximal rainbow matching in $G$ and for contradiction sake, suppose that $|M| \leq n-r-1$. By Corollary \ref{corollary}, we know that provided that $r$ is sufficiently large and $n$ is sufficiently large in terms of $r$, there is a $9$-auxiliary matching $N$ for $M$ of size at least $0.9n$ such that $N_{1/4} = \emptyset$. By deleting at most $0.1n$ edges from $M_N$ whose colours are not in $C'$, we
can further redefine $N$ so that $C_N \subseteq C'$ and still $|N| \geq 0.8n> 2n/3$. Then, recalling the definition of $N_{1/4}$ and taking into account the previous modification made to $N$, we now have that for every $e \in M_N$, the edge $v_e m(x_e)$ is repeated in at least $(1/4) \cdot 0.9n - 0.1n \geq 0.1n$ many colours in $C_N$. Now, define the matching $L := \{v_e m(x_e): e \in M_N\} \cup (M \setminus M_N)$ and note the following (the reader might want to refer to Figures \ref{fig4} and \ref{fig5} for an illustration).
\begin{claim*}\label{claim42}
There is no $C_N$-coloured edge contained in $V \setminus V(L)$ and $L$ contains no $C_N$-rainbow horn.
\end{claim*} 
\begin{proof}
For the first part, let $e'$ be a $C_N$-coloured edge contained in $V \setminus V(L)$ and let $e_1, \ldots, e_i$ denote the edges of $M_N$ which intersect $e'$. By the condition we have on the edges $v_e m(x_e)$ (for $e \in M_N$) and since $i \leq 2$, we can pick distinct colours in $C_N \setminus \{c(e')\}$ for the edges $v_{e_1} m(x_{e_1}), \ldots, v_{e_i} m(x_{e_i})$ so that 
$\{e', v_{e_1} m(x_{e_1}), \ldots, v_{e_i} m(x_{e_i})\}$ is a rainbow matching of size $i+1$ contradicting Lemma \ref{lema} with $M'=\{e_1, \dots, e_i\} \subseteq M$.

For the second part, suppose for some $e \in M_N$, there exist distinct vertices $z_1,z_2 \notin V(L)$ such that the edges $v_e z_1$ and $m(x_e) z_2$ are distinctly coloured in $C_N$. Let $c_1,c_2$ denote these colours. Let also $e_1, \ldots, e_i$ denote the edges of $M_N$ which intersect $\{z_1,z_2\}$. As before, since $i \leq 2$, we can pick distinct colours in $C_N \setminus \{c_1,c_2\}$ for the edges $v_{e_1}m(x_{e_1}), \ldots, v_{e_i}m(x_{e_i})$ so that $\{v_e z_1,m(x_e) z_2,v_{e_1}m(x_{e_1}), \ldots, v_{e_i}m(x_{e_i})\}$ is a rainbow matching of size $i + 2$ disjoint to the vertices in $V(M \setminus \{e,e_1, \ldots, e_i\}) \cup \{v_{f} : f \in M_N \setminus \{e,e_1, \ldots, e_i\}\} \cup \{x_{f} : f \in M_N \setminus \{e,e_1, \ldots, e_i\}\}$. Since $N$ is a $9$-auxiliary matching,
this contradicts Lemma \ref{lema} for $M'=\{e, e_1, \dots, e_i\} \subset M$ and $t=9$. A similar analysis can be done when the horn is an edge in $M \setminus M_N$.
\end{proof}

\begin{figure}[h]
\RawFloats
\begin{minipage}[t]{0.5\textwidth}
\centering
\captionsetup{width=\textwidth}
\begin{tikzpicture}[scale =1, xscale=1, yscale=0.9]

\node[scale=1] (M) at(-1.1,0.1) {$e'$};

\node[scale=1] (M) at(0,0.7) {$e_1$};

\node[scale=1] (M) at(0,-1.3) {$e_2$};

\draw (-0.8,1) -- (1,1);
\draw (-0.8,1) -- (1,1);
\draw (-0.8,-1) -- (1,-1);
\draw (-0.8,-1) -- (1,-1);
\draw[blue!100] (-0.8,1) -- (-0.8,-1);
\draw[blue!100] (-0.8,1) -- (-0.8,-1);
\draw[blue!100] (-0.8,1) -- (-0.8,-1);
\draw[blue!100] (-0.8,1) -- (-0.8,-1);

\path[gray!100] (-1.8,-1) edge [bend left=35] (1,-1);
\path[gray!100] (-1.8,-1) edge [bend left=15] (1,-1);
\path[gray!100] (-1.8,-1) edge [bend left=25] (1,-1);
\path[gray!100] (-1.8,-1) edge [bend left=35] (1,-1);
\path[gray!100] (-1.8,-1) edge [bend left=15] (1,-1);
\path[gray!100] (-1.8,-1) edge [bend left=25] (1,-1);

\path[gray!100] (-1.8,1) edge [bend left=35] (1,1);
\path[gray!100] (-1.8,1) edge [bend left=15] (1,1);
\path[gray!100] (-1.8,1) edge [bend left=25] (1,1);
\path[gray!100] (-1.8,1) edge [bend left=35] (1,1);
\path[gray!100] (-1.8,1) edge [bend left=15] (1,1);
\path[gray!100] (-1.8,1) edge [bend left=25] (1,1);

\draw[red!100] (-1.8,-1) -- (-0.8,-1); 
\draw[red!100] (-1.3,-1) ellipse (0.5 and 0.08); 
\draw[red!100] (-1.8,1) -- (-0.8,1); 
\draw[red!100] (-1.3,1) ellipse (0.5 and 0.08); 

\draw[white!100] (-0.7,-2.5) -- (0,-2.5);;

\node[scale=3] at(-0.8,1) {$.$};
\node[scale=3] at(-0.8,-1) {$.$};

\end{tikzpicture}
 
\caption{A $C_N$-coloured edge outside of $L$.}
\label{fig4}
\end{minipage}\hfill
\begin{minipage}[t]{0.5\textwidth}
\centering
\captionsetup{width=1\textwidth}
\begin{tikzpicture}[scale = 0.95, xscale=1, yscale=0.9]

\draw (-0.8,2) -- (1,2);
\draw (-0.8,2) -- (1,2);
\draw (-0.8,0) -- (1,0);
\draw (-0.8,0) -- (1,0);
\draw (-0.8,-2) -- (1,-2);
\draw (-0.8,-2) -- (1,-2);

\node[scale=1] at(0,-2.3) {$e$};
\node[scale=0.8] at(-0.65,-0.35) {$z_1$};
\node[scale=0.8] at(-0.95,1.65) {$z_2$};
\node[scale=0.8] at(-1.8,-2.35) {$v_e$};
\node[scale=0.8] at(1,-2.35) {$m(x_e)$};

\path[gray!100] (-1.8,-2) edge [bend left=35] (1,-2);
\path[gray!100] (-1.8,-2) edge [bend left=15] (1,-2);
\path[gray!100] (-1.8,-2) edge [bend left=25] (1,-2);
\path[gray!100] (-1.8,-2) edge [bend left=35] (1,-2);
\path[gray!100] (-1.8,-2) edge [bend left=15] (1,-2);
\path[gray!100] (-1.8,-2) edge [bend left=25] (1,-2);

\path[gray!100] (-1.8,0) edge [bend left=35] (1,0);
\path[gray!100] (-1.8,0) edge [bend left=15] (1,0);
\path[gray!100] (-1.8,0) edge [bend left=25] (1,0);
\path[gray!100] (-1.8,0) edge [bend left=35] (1,0);
\path[gray!100] (-1.8,0) edge [bend left=15] (1,0);
\path[gray!100] (-1.8,0) edge [bend left=25] (1,0);

\path[gray!100] (-1.8,2) edge [bend left=35] (1,2);
\path[gray!100] (-1.8,2) edge [bend left=15] (1,2);
\path[gray!100] (-1.8,2) edge [bend left=25] (1,2);
\path[gray!100] (-1.8,2) edge [bend left=35] (1,2);
\path[gray!100] (-1.8,2) edge [bend left=15] (1,2);
\path[gray!100] (-1.8,2) edge [bend left=25] (1,2);

\draw (-0.8,2) -- (1,2);
\draw (-0.8,2) -- (1,2);
\draw (-0.8,0) -- (1,0);
\draw (-0.8,0) -- (1,0);
\draw (-0.8,-2) -- (1,-2);
\draw (-0.8,-2) -- (1,-2);

\draw[red!100] (-1.8,2) -- (-0.8,2); 
\draw[red!100] (-1.3,2) ellipse (0.5 and 0.08); 
\draw[red!100] (-1.8,0) -- (-0.8,0); 
\draw[red!100] (-1.3,0) ellipse (0.5 and 0.08); 
\draw[red!100] (-1.8,-2) -- (-0.8,-2); 
\draw[red!100] (-1.3,-2) ellipse (0.5 and 0.08);

\draw[green!=100] (-1.8,-2) -- (-0.8,0);
\draw[green!=100] (-1.8,-2) -- (-0.8,0);
\draw[green!=100] (-1.8,-2) -- (-0.8,0);
\draw[green!=100] (-1.8,-2) -- (-0.8,0);
\draw[blue!=100] (1,-2) -- (-0.8,2);
\draw[blue!=100] (1,-2) -- (-0.8,2);
\draw[blue!=100] (1,-2) -- (-0.8,2);
\draw[blue!=100] (1,-2) -- (-0.8,2);

\node[scale=3] at(-0.8,2) {$.$};
\node[scale=3] at(-0.8,0) {$.$};
\node[scale=3] at(1,-2) {$.$};
\node[scale=3] at(-1.8,-2) {$.$};

\end{tikzpicture}
\caption{A $C_N$-rainbow horn in $L$.}
\label{fig5}
\end{minipage}
\end{figure}

\noindent Let now $c \in C_N \subseteq C'$. Note that since it has at least $3(n-r)-2 \geq 3|L|+1$ vertices in its colour class, by the claim above and Lemma \ref{triangles}, there must exist a $c$-coloured matching $L_c \subseteq E[V(L), V \setminus V(L)]$ of size $|L|+1$. Also, using the claim above and Observation \ref{observ}, we must have that any edge $e' \in L$ is a $c$-horn (in $L$) for at most two colours $c \in C_N$, and thus, in general, every edge in $L$ intersects at most $|C_N|+2$ many edges belonging to $\bigcup_{c \in C_N} L_c$. Furthermore, suppose that $e' \in L$ is of the form $e' = v_e m(x_e)$, for some $e \in M_N$. If it is a $c$-horn (in $L$) for some $c \in C_N$, then recall that $e'$ is repeated in at least $n/10-1 > n/20$ many colours $c' \in C_N \setminus \{c\}$. Since each such colour $c'$ is a disjoint union of non-trivial cliques, no edge of $L_{c'}$ can intersect $e'$, otherwise this would imply that $e'$ is a $C_N$-rainbow horn (with the two implicit colours being $c,c'$), which contradicts the above claim. Hence, the edge $e'$ intersects at most $2+|C_N| - n/20 \leq |C_N|$ many edges from $\bigcup_{c \in C_N} L_c$. Trivially, if $e'$ is not a $c$-horn for any $c \in C_N$, this also holds. Therefore, by double-counting we must have
$$|C_N| (|L|+1) = \sum_{c \in C_N} |L_c| \leq |M_N| \cdot |C_N| + |M \setminus M_N| \cdot (|C_N| + 2)$$
implying that $|C_N| \leq 2|M \setminus M_N| \leq 2n - 2|C_N|$, which is a contradiction since $|C_N| = |N|> 2n/3$. 
\end{proof}
Following the ideas briefly discussed in the outline of the proof,  we can now establish our main result. 
\begin{proof}[ of Theorem \ref{mainthm}]
Let $G$ be a $(n,3n-2)$-multigraph. Let $r$ be sufficiently large and $n$ sufficiently large in terms of $r$, so that in particular, Lemma \ref{keylemma} holds. Note that we can iteratively apply Theorem \ref{strongboundedmult} with $\delta = 1/10r$ in order to find a matching $L$ of size $r$ such that each edge in it is repeated in at least $(1-1/10r)n$ many colours. Let $C'$ be the set of colours which are repeated in all edges of $L$ and note that by a union bound, we have $|C'| \geq 9n/10$. Since each colour class in $G$ is a disjoint union of triangles and edges, $G':= G - V(L)$ is now a $(n,3n-10r)$-multigraph such that every colour in $C'$ has at least $3(n-r)-2$ vertices in its colour class. Therefore, Lemma \ref{keylemma} implies that $G'$ has a rainbow matching of size at least $n-r$. 

In order to finish the proof, we only need to show that there exists a maximal rainbow matching $M$ in $G'$ such that the colours not used in it are contained in $C'$. Indeed, if this is the case, since $|M| \geq n-r$, we can then pick the colours of the edges of $L$, so that $M \cup L$ contains a rainbow matching of size $n$, and we are done. Take then a maximal rainbow matching $M'$ in $G'$ and suppose there is some colour $c \notin C'$ which is not used in $M'$. Since $M'$ is maximal, there is no $c$-coloured edge in $G'$ which is contained outside $M'$ and thus, by Lemma \ref{triangles}, there is a $c$-coloured matching in $E[V(M'), V(G') \setminus V(M')]$ of size $3n-10r-|M'| \geq n-10r > 2n/3$. In particular, this matching then intersects at least $n/3$ edges of $M'$ - and for each such edge, note we can switch it with the edge in the $c$-coloured matching which touches it. Thus, there exist at least $n/3$ colours $c'$ which are used in $M'$ for which there is a maximal rainbow matching $M''$ in $G'$ such that $C(M'') = (C(M') \setminus \{c'\}) \cup \{c\}$. As $|C'| \geq 9n/10$, there is always such a colour $c'$ which also belongs to $C'$. We can now take $M''$ and repeat the same operation until we have a maximal rainbow matching $M$ in $G'$ such that all the colours not used in it belong to $C'$.
\end{proof}

\section{Concluding remarks}
Although we have resolved the Grinblat problem for all sufficiently large $n$, there are still some related open questions which could be of interest. Firstly, one might want to understand what happens for small $n \geq 4$. Our proof yields the validity of Grinblat's conjecture starting with moderately large $n$, but it seems plausible that some of our ideas could help in this direction. Secondly, one can consider stability-type questions: Is the example consisting of a disjoint union of $n-1$ triangles repeated in each of the $n$ colours the only $(n,3n-3)$-multigraph without a rainbow matching of size $n$? How close to this example are all $(n,v)$-multigraphs with no rainbow matching of size $n$ and $v$ close to $3n$? 

Finally, another interesting direction is to gain a better understanding of how Grinblat's problem varies with edge-multiplicity restrictions. Precisely, one would look to answer the following question.
Given $n$ and $m \leq n$, what is the minimal $v_m$ such that every $(n,v_m)$-multigraph with edge-multiplicity at most $m$ contains a rainbow matching of size $n$?
The difficulty of determining or estimating the value $v_m$ will vary with $m$. For example, note that Theorem \ref{mainthm} states that $v_n = 3n-2$, whereas determining $v_1$ seems more difficult. The first author together with Yepremyan \cite{correia2020full} showed that $v_1 = 2n + o(n)$. This might suggest that $v_1 = 2n + c$ for some constant $c$. However, recall that the famous Ryser-Brualdi-Stein conjecture can be formulated as stating that every properly $n$-edge-coloured $K_{n,n}$ contains a rainbow matching of size $n-1$. Let $G$ be the graph formed by taking a disjoint union of $c/2$ stars of size $n$, each having edges of all the $n$ colours. Note that the graph formed by the disjoint union of $G$ and a properly $n$-edge-coloured $K_{n,n}$ is a $(n,2n+c)$-multigraph with multiplicity at most $1$. Therefore, if one proves that $v_1 = 2n+c$, then there exists a rainbow matching of size $n$ in this graph. This will imply that $K_{n,n}$ contains a rainbow matching of size $n-c/2$, which would greatly improve on the best known bound for the Ryser-Brualdi-Stein conjecture from \cite{keevash2020new}.

Recently, the two authors together with Pokrovskiy \cite{CPS2021+} showed that $v_{m} \leq 2n + 2m + O(n/(\log n)^{1/4})$. This substantially improves the result in \cite{correia2020full} and is asymptotically tight for $m = \eps n$ and $\log^{-1/4} n \ll \eps \ll 1$. Indeed, one can construct examples (see \cite{CPS2021+}) of $(n,2n+2\eps n - O(\eps^2 n))$-multigraphs with edge-multiplicity at most $\eps n$ and no rainbow matching of size $n$. However, this result does not give much information about what occurs when $m$ is very small, neither when $m$ is close to $n$.

\vspace{0.4cm}
\noindent
{\bf Acknowledgements.} The authors would like to thank Alexey Pokrovskiy for stimulating discussions on the topic. The first author would also like to thank Liana Yepremyan for introducing him to the problem and for valuable discussions.

\end{document}